\documentclass{article}
\usepackage[T2A]{fontenc}
\usepackage[utf8]{inputenc}
\usepackage[english]{babel}
\usepackage[tbtags]{amsmath}
\usepackage{amsfonts,amssymb,mathrsfs,amscd}
\usepackage{graphicx}

\usepackage{float}

\numberwithin{equation}{section}
\newtheorem{theorem}{Theorem}
\newtheorem{lemma}{Lemma}[section]

\newtheorem{proof}{Proof}
\newtheorem{remark}{Remark}

\begin{document}

\begin{center}
{\bf Sidon-type inequalities for $p$-adic analogues of Rademacher chaos}
\end{center}

\begin{center}
A.~D.~Kazakova
\end{center}


\section{Introduction}

A system of functions $(\gamma_k)_{k \in \mathbb{N}}$ is called a \textit{Sidon system} if the following Sidon inequality holds:
\begin{equation} \label{Sidon}
\|(a_k)_{k=1}^N\|_{\ell_1} \le C \left\|\sum_{k=1}^N a_k \gamma_k\right\|_{L^\infty}, \qquad C>0 \; \text{is independent of } a_k \in \mathbb{C}.
\end{equation}
S. Sidon \cite{SidonTrig, Sidon} proved that Hadamard lacunary subsystems of the trigonometric or Walsh systems are Sidon systems. One of the subsequent directions in generalizing the Sidon inequality is related to its application to the estimation of entropy numbers and Kolmogorov widths \cite{KasTem98}, \cite{Rad22}.

Sidon's proof relies on the use of Riesz products, which have become an important tool in the theory of orthogonal series (see, e.g., \cite[Ch. 9, \S 4]{Kashin-Saakyn} and \cite[Ch. VII]{Blei}). A classical example of a Sidon system is the Rademacher system \cite{Ast-book}, which forms a Hadamard lacunary subsystem of the Walsh system. In the space $L^\infty$, the system consisting of $d$-fold products of distinct Rademacher functions (the \textit{$d$-th order Rademacher chaos}) behaves quite differently: it is no longer a Sidon system, but nevertheless forms a $\frac{2d}{d+1}$-Sidon system \cite[Ch. VII]{Blei}. The latter means that inequality \eqref{Sidon} holds with the $\ell_1$-norm on the left-hand side replaced by the $\ell_{\frac{2d}{d+1}}$-norm, and the constant $\frac{2d}{d+1}$ cannot be improved.

Let $p \ge 2$ be fixed. The \textit{generalized Rademacher functions} on $[0,1)$ are defined by the formula
$$
R_k (x) := \exp\left(\frac{2 \pi i}{p} \lfloor x p^{k+1} \rfloor \right), \qquad k \in \mathbb{N}_0.
$$
Consider the system
\begin{equation} \label{chaos}
    A_d:= \{R_{k_1}^{l_1} \cdots R_{k_d}^{l_d} : k_1 < \ldots < k_d, \; l_1, \ldots, l_d \in \{1, \ldots, p-1\} \},
\end{equation}
which represents the most general extension of the $d$-th order Rademacher chaos to the $p$-ary case. The system of functions \eqref{chaos} is a subsystem of the Vilenkin--Сrestenson system (which generalizes the Walsh system from the case $p=2$), and under the Paley enumeration it corresponds to the indices:
\begin{equation} \label{n}
    N_d:= \{ n \in \mathbb{N} : n = l_1 p^{k_1} + \ldots + l_d p^{k_d}, \quad k_1 < \ldots < k_d, \; l_1, \ldots, l_d \in \{1, \ldots, p-1\} \}.
\end{equation}

According to a result by R. Blei \cite[Ch. VII, \S 12]{Blei}, the tetrahedral $p$-ary Rademacher chaos
\[
    \{R_{k_1} \cdots R_{k_d} : k_1 < \ldots < k_d\},
\]
which forms a subsystem of \eqref{chaos}, is a $\frac{2d}{d+1}$-Sidon system.

We show that the system \eqref{chaos} is also a $\frac{2d}{d+1}$-Sidon system.
The main tool of the proof is the application of the analogs of the Riss--Rademacher products, as well as the ideas from the paper \cite{KazPlo25}.

\section{Constructions of measures}
The \textit{Riesz--Rademacher product} is defined as a measure $\rho$ on the interval $[0,1]$ that arises as the weak-$^*$ limit of a sequence of measures with densities with respect to the Lebesgue measure:
\begin{equation} 
 \rho_n(x) = \prod_{k=1}^{n} \bigl(1 + a_k r_k(x)\bigr),
\end{equation}
where the coefficients $a_k \in \mathbb{R}$ satisfy $|a_k| \le 1$.

The Riesz--Rademacher product can be generalized to the case of generalized Rademacher functions as follows:
\begin{equation} \label{28042248}
    \rho = \operatorname{weak-}^*\lim_{n\rightarrow \infty} \rho_n, \qquad  \rho_n= \prod_{k=1}^{n} \left(1+\frac{1}{2}\Bigl[a_k R_k^{j_k}(x) + \overline{a_k R_k^{j_k}(x)}\Bigr]\right),
\end{equation}
where $\{j_k\}_{k=1}^{\infty}$ is a sequence taking values in $\{1, \dots, p-1\}$ and $a_k \in \mathbb{C}$ satisfy $|a_k| \le 1$. 

Expression \eqref{28042248} is well-defined, since
\begin{equation} \label{ravenstvo}
\|\rho_n\|_{L_1} = \int_0^1 |\rho_n(x)|\,\mathrm{d}\mu = \int_0^1 \rho_n(x)\,\mathrm{d}\mu =1,
\end{equation}
and by the Banach--Alaoglu theorem, there exists a positive Borel measure $\rho$ defined by \eqref{28042248} such that $\|\rho\| =1$, where $\|\cdot\|$ denotes the total variation of the measure.

\begin{lemma} \label{lemma-1}
    Let $J = (j_k)_{k=0}^{\infty}$ be a fixed sequence of integers from the set $\{1, \dots, p-1\}$. Then there exists a measure $\nu$ of bounded variation such that the following holds for its Fourier--Vilenkin--Crestenson coefficients:
\[
\begin{cases}
    \widehat{\nu}_n = 1, \quad \text{if $n\in N_d$ and $l_1 = j_{k_1}, \ldots, l_d = j_{k_d}$}, \\
    \widehat{\nu}_n = 0, \quad \text{if $n\in N_d$ and $\exists \; q\colon \;l_q \neq j_{k_q}$}.
\end{cases}
\]
\end{lemma}
\begin{proof}
Let $a := \exp\left(\frac{2 \pi i}{2d+1}\right)$ be a $(2d+1)$-th root of unity. Consider the measure $\rho$ defined by \eqref{28042248} such that
\begin{equation*}
    a_k = a, \quad\text{if $R_k^{j_k} \neq \overline{R_k^{j_k}}$}; \; \qquad a_k = 1, \quad\text{if $R_k^{j_k} = \overline{R_k^{j_k}}$.}
\end{equation*}
Let us compute the coefficients of the measure $\rho$:
\[
\widehat{\rho}_n =  \int_0^1 \overline{R_{k_1}^{l_{1}}\cdot \ldots \cdot R_{k_d}^{l_{d}}}d \rho, \quad n = l_1p^{k_1}+\ldots+l_dp^{k_d}.
\]
Then for $n\in N_d$:
\[
\widehat{\rho}_n \in {\cal M}_d, \quad  {\cal M}_d := \bigsqcup_{l=0}^{d}\left\{\frac{a^{d-l}}{2^{d-l}}, \frac{a^{d-l-2}}{2^{d-l}}, \ldots,  \frac{a^{-(d-l)}}{2^{d-l}} \right\} \bigsqcup \{0\}.
\]
By construction, for indices $n \in N_d$ satisfying $l_1 = j_{k_1}, \ldots, l_d = j_{k_d}$, the corresponding coefficients belong to the set
\[
    {\cal J}_d := \left\{\frac{a^d}{2^d}, \frac{a^{d-1}}{2^{d-1}}, \ldots, 1 \right\} \subset {\cal M}_d.
\]
We seek a measure $\nu$ of the form
\begin{equation} \label{mera}
\nu = c_0 \delta_0 +c_1 \rho + \ldots + c_{\frac{(d+1)(d+2)}{2}}  \underbrace{\rho * \ldots *\rho}_{\text{$\frac{(d+1)(d+2)}{2}$ times}}, \quad \delta_0 \text{is the Dirac delta measure at zero}, \; c_i \in \mathbb{C},
\end{equation}
and $*$ means convolution of measures.
Then the Fourier coefficients of measure $\nu$ are expressed as a polynomial of degree $\frac{(d+1)(d+2)}{2}$:
\begin{equation} \label{coef}
    \widehat{\nu}_n = c_0 +c_1 \widehat{\rho}_n + \ldots +  c_{\frac{(d+1)(d+2)}{2}} (\widehat{\rho}_n)^{\frac{(d+1)(d+2)}{2}}. 
\end{equation}
Let \[T := \begin{bmatrix}
    1 & 0 & 0 & \cdots & \cdots & 0 \\
    1 & t_{00} & t_{00}^2 & \cdots & \cdots & t_{00}^{\frac{(d+1)(d+2)}{2}} \\
    \vdots & \vdots & \vdots  & \vdots & \vdots  & \vdots \\
    1 & t_{0d} & t_{0d}^2 &  \cdots & \cdots & t_{0d}^{\frac{(d+1)(d+2)}{2}} \\
    \vdots & \vdots & \vdots & \vdots & \vdots  & \vdots \\
    1 & t_{d0} & t_{d0}^2 &  \cdots & \cdots & t_{d0}^{\frac{(d+1)(d+2)}{2}}
\end{bmatrix}, \qquad \mathbf{b} := \begin{bmatrix}
    0 \\
    \mathbf{e}_{0}  \\
        \vdots \\
    \mathbf{e}_{l} \\
    \vdots \\
    \mathbf{e}_{d} 
\end{bmatrix},  \qquad \mathbf{c} := \begin{bmatrix}
    c_0  \\
    c_1 \\
    \vdots \\
    c_{\frac{(d+1)(d+2)}{2}}
\end{bmatrix},
\]
where $t_{li} := \frac{a^{2i-(d-l)}}{2^{d-l}}$, for $l \in 0:d$ и $i \in 0:d-l$ и $\mathbf{e}_{l} := (0, 0, \ldots, 0, 1)^{T} \in \mathbb{R}^{d-l+1}$.

We obtain a system of equations for the coefficients of the measure $\nu$: $T\mathbf{c} = \mathbf{b}$, where $T$ is a Vandermonde matrix. It is straightforward to verify that the parameters $t_{lj}$ are chosen such that the matrix is non-singular:
\[
\det T = \prod (t_{lj} - t_{ki}),
\]
where the product is taken over all admissible pairs of indices $(l, i)$ and $(k, j)$. Suppose there exist two pairs $(l, i)$ and $(k, j)$ such that
\[
t_{lj} - t_{ki} = 0.
\]
This leads to two cases:

1) $k \neq l$, then
\[\frac{a^{2i-(d-l)}}{2^{d-l}} = \frac{a^{2i-(d-k)}}{2^{d-k}}, \]
i.e. $2^{k-l}a^{2i-(d-l)} = a^{2i-(d-k)}$, but this is impossible since $|a|=1$.

2) $k = l$, then
\[
t_j - t_i = \frac{1}{2^{d-l}}(a^{2j-(d-l)} - a^{2i - (d-l)}) = 0,  
\]
if 
\[
2j - (d-l) = 2i - (d-l) \qquad (mod \;  2d+1).
\]
\[2(j-i) = 0 \qquad (mod \;  2d+1). \]
Then, 
\[j-i = 0 \qquad (mod \;  2d+1), \]
but this is impossible since $1 \le |j-i| \le d$ for distinct $i$ and $j$.

The measure $\nu$ satisfies the required conditions, and its total variation is bounded from above by Young's inequality:
\begin{equation} \label{var-nu}
    \|\nu\|  \le \|\mathbf{c}\|_{l_1}.
\end{equation}
\end{proof}

The following lemma generalizes Lemma 22 from \cite[Ch. VII]{Blei}.
\begin{lemma} \label{lemma-2}
Let $s \le d$ be fixed. Then there exists a measure $\nu$ such that the following holds for its Fourier--Vilenkin--Crestenson coefficients:
\[
\begin{cases}
    \widehat{\nu}_n =1, \quad n \in N_s, \\
    \widehat{\nu}_n =0, \quad n \in \bigsqcup\limits_{j\le d\colon j \neq s} N_j.
\end{cases}
\] 
\end{lemma}
\begin{proof}
    Consider the measure
    \[
    \rho= weak^*\lim_{n \rightarrow\infty} \prod_{j=0}^{n} \left(1 + \frac{R_j + R_j^2 + \ldots + R_j^{p-1}}{p}\right), \qquad \|\rho\|=1.
    \]
    Coefficients of this measure 
    \[
    \widehat{\rho}_n= 1/p^s, \quad \text{for $n\in N_s$}.
    \]
    There exists a polynomial $P$ of degree $d$ such that
    \[
    P(0) = 0; \; P\left(\frac{1}{p^s}\right) = 1; \; P\left(\frac{1}{p^j}\right) = 0, \, \text{for $j\in \{1, \ldots, d\} \setminus \{s\}$}.
    \]
    Then the desired measure $\nu$ is given by
    \[
    \nu = c_1\rho + \ldots + c_d \underbrace{\rho * \ldots * \rho}_{\text{$d$ times}},
    \]
where $c_i$ are the coefficients of the polynomial $P$. The total variation of the measure $\nu$ satisfies an estimate analogous to \eqref{var-nu}.
\end{proof}
\section{Main results}
\begin{theorem} \label{nadolyche}
    For any sequence $J = (j_i)_{i=1}^{\infty}$ with $j_i \in \{1, \dots, p-1\}$, the system of functions \[
     \{R_{k_1}^{j_{k_1}} \cdot \ldots \cdot R_{k_d}^{j_{k_d}}, k_1 <\ldots<k_d\}.
    \] is a $\frac{2d}{d+1}$-Sidon system:
    \[
            \|(C_{k_1, \ldots, k_d})_{{0\le k_1 < \ldots < k_d \le N}} \|_{l_{2d/(d+1)}} \le 
    \]
    \begin{equation*}
\le C \bigl\|\sum_{0\le k_1 < \ldots < k_d \le N}  C_{k_1, \ldots, k_d} R^{j_{k_1}}_{k_1}(x) \cdot \ldots \cdot R^{j_{k_d}}_{k_d}(x) \bigl\|_{L_{\infty}}, \quad C_{i_1,\ldots, i_s}^{j_1, \ldots, j_s} \in \mathbb{C},
    \end{equation*} 
where the constant $C$ is independent of both the polynomial and the sequence $J$.

\end{theorem} 
\begin{proof}
For a set $J = (j_k)_{k=0}^{\infty}$ and for a fixed $y \in [0,1],$ consider the measure $\rho_y$ defined by formula \eqref{28042248}, constructed from this set, such that
\[
 a_k = r_k(y), \quad \text{if $R_k^{j_k} \neq \overline{R_k^{j_k}}$}; \qquad a_k = \frac{r_k(y)}{2}, \quad \text{if $R_k^{j_k} = \overline{R_k^{j_k}}$}.
\]
Let us compute the coefficients of the measure  $\rho_y$:
\[
\widehat{(\rho_y)}_n = \frac{r_{k_1}(y)\cdot \ldots \cdot r_{k_d}(y)}{2^d},\qquad \text{where} \;   
    n = j_{k_1}p^{k_1} + \ldots + j_{k_d}p^{k_d}.\]
Note that the total variation of the measure $\rho_y$ satisfies $\|\rho_y\| = 1$.

Consider
\[
\left[\sum_{k_1<\ldots<k_d} C_{k_1,\ldots, k_d} R_{k_1}^{j_{k_1}}(x)\ldots R_{k_d}^{j_{k_d}}(x)\right] * \rho_y = \]
\[\int_0^1 \left[\sum_{k_1<\ldots<k_d} C_{k_1,\ldots, k_d} R_{k_1}^{j_{k_1}}(x\ominus z)\ldots R_{k_d}^{j_{k_d}}(x\ominus z)\right] d\rho_y(z) = 
\]
\[
 \sum_{k_1<\ldots<k_d} C_{k_1,\ldots, k_d} R_{k_1}(x)\ldots R_{k_d}(x)\int_0^1  \overline{R_{k_1}(z)\ldots R_{k_d}(z)} d\rho_y(z) = 
\]
\[
 \sum_{k_1<\ldots<k_d} C_{k_1,\ldots, k_d} R_{k_1}^{j_{k_1}}(x)\ldots R_{k_d}^{j_{k_d}}(x)\widehat{(\rho_y)}_{j_{k_1}p^{k_1} + \ldots + j_{k_d}p^{k_d}} = 
\]
\[
 \frac{1}{2^d}\sum_{k_1<\ldots<k_d} C_{k_1,\ldots, k_d} R_{k_1}^{j_{k_1}}(x)\ldots R_{k_d}^{j_{k_d}}(x)  r_{k_1}(y)\cdot \ldots \cdot r_{k_d}(y).
\]

Then, applying Young's inequality for convolutions, we obtain:
\[
\biggl\|  \frac{1}{2^d}\sum_{k_1<\ldots<k_d} C_{k_1,\ldots, k_d} R_{k_1}^{j_{k_1}}(x)\ldots R_{k_d}^{j_{k_d}}(x)  r_{k_1}(y)\cdot \ldots \cdot r_{k_d}(y) \biggl\|_{L_\infty(x)} \le 
\]
\[
\biggl\| \sum_{k_1<\ldots<k_d} C_{k_1,\ldots, k_d} R_{k_1}^{j_{k_1}}(x)\ldots R_{k_d}^{j_{k_d}}(x)\biggl\|_{L_\infty(x)} .
\]
Since this is true for all $y$, the following is also true
\[\biggl\|  \frac{1}{2^d}\sum_{k_1<\ldots<k_d} C_{k_1,\ldots, k_d} R_{k_1}^{j_{k_1}}(x)\ldots R_{k_d}^{j_{k_d}}(x)  r_{k_1}(y)\cdot \ldots \cdot r_{k_d}(y)\biggl\|_{L_\infty(x) \times  L_\infty(y)}\le 
\]
\[
\biggl\| \sum_{k_1<\ldots<k_d} C_{k_1,\ldots, k_d} R_{k_1}^{j_{k_1}}(x)\ldots R_{k_d}^{j_{k_d}}(x)\biggl\|_{L_\infty(x)} .
\]
Thus, for almost all $x^*$:
\[\biggl\|\sum_{k_1<\ldots<k_d} C_{k_1,\ldots, k_d} R_{k_1}^{j_{k_1}}(x^*)\ldots R_{k_d}^{j_{k_d}}(x^*)  r_{k_1}(y)\cdot \ldots \cdot r_{k_d}(y) \biggl\|_{L_\infty(y)}\le 
\]
\[
2^d\biggl\| \sum_{k_1<\ldots<k_d} C_{k_1,\ldots, k_d} R_{k_1}^{j_{k_1}}(x)\ldots R_{k_d}^{j_{k_d}}(x)\biggl\|_{L_\infty(x)} .
\]
On the other hand, by the $\frac{2d}{d+1}$-Sidon inequality for the classical $d$-th order Rademacher chaos,
\[
 \biggl(\sum_{k_1< \ldots < k_d} \bigl|C_{k_1, \ldots, k_d} R_{k_1}^{j_{k_1}}(x^*)\ldots R_{k_d}^{j_{k_d}}(x^*)\bigl|^{\frac{2d}{d+1}} \biggl)^{\frac{d+1}{2d}} \le \]
\[
 C \biggl\|\sum_{k_1<\ldots<k_d} C_{k_1,\ldots, k_d} R_{k_1}^{j_{k_1}}(x^*)\ldots R_{k_d}^{j_{k_d}}(x^*)  r_{k_1}(y)\cdot \ldots \cdot r_{k_d}(y) \biggl\|_{L_\infty(y)}.
\]
Noting that
\[
\bigl|C_{k_1, \ldots, k_d} R_{k_1}^{j_{k_1}}(x^*)\ldots R_{k_d}^{j_{k_d}}(x^*)\bigl| = \bigl|C_{k_1, \ldots, k_d}\bigl|,
\]
obtain
 \[ \biggl(\sum_{k_1< \ldots <k_d} \bigl|C_{k_1, \ldots, k_d}\bigl|^{\frac{2d}{d+1}}\biggl)^{\frac{d+1}{2d}} \le 2^dC \biggl\| \sum_{k_1<\ldots<k_d} C_{k_1,\ldots, k_d} R_{k_1}^{j_{k_1}}(x)\ldots R_{k_d}^{j_{k_d}}(x)\biggl\|_{L_\infty(x)}.
 \]
 \end{proof}
 
\begin{theorem} \label{mainS} 
The system $A_d$ is a $\frac{2d}{d+1}$-Sidon system, i.e., the following inequality holds:
    \[
\biggl\|\left(C^{l_1, \ldots, l_d }_{k_1, \ldots, k_d}\right)_{{0\le k_1 < \ldots < k_d \le N}, l_1, \ldots, l_d \in \{1, \dots, p-1\} } \biggl\|_{l_{2d/(d+1)}} \le 
    \]
    \begin{equation*}
\le C \biggl\|\sum_{0\le k_1 < \ldots < k_d \le N} \; \sum_{l_1, \ldots, l_d \in \{1, \dots, p-1\}} C^{l_1, \ldots, l_d }_{k_1, \ldots, k_d} R^{l_1}_{k_1}(x) \cdot \ldots \cdot R^{l_d}_{k_d}(x) \biggl\|_{L_{\infty}},  \quad C_{i_1,\ldots, i_s}^{j_1, \ldots, j_s} \in \mathbb{C},
    \end{equation*}
where the constant $C$ is independent of the polynomial.
\end{theorem}
\begin{proof}
We denote  
\[Q = \sum_{0\le k_1 < \ldots < k_d \le N} \; \sum_{l_1, \ldots, l_d \in \{1, \dots, p-1\}} C^{l_1, \ldots, l_d }_{k_1, \ldots, k_d} R^{l_1}_{k_1}(x) \cdot \ldots \cdot R^{l_d}_{k_d}(x).  \]
For a fixed sequence $J = (j_k)_{k=0}^{\infty}$, consider the measure $\nu$ constructed in Lemma \ref{lemma-1}. We have the following chain of equalities: 
\[
Q * \nu = \int_0^1 Q(x\ominus y) d\nu(y) = 
\]
\[
= \sum_{0\le k_1 < \ldots < k_d \le N} \; \sum_{l_1, \ldots, l_d \in \{1, \dots, p-1\}} R^{l_1}_{k_1}(x) \cdot \ldots \cdot R^{l_d}_{k_d}(x) \int_{0}^{1} \overline{R^{l_1}_{k_1}(y) \cdot \ldots \cdot R^{l_d}_{k_d}(y)} d \nu(y) = 
\]
\[
= \sum_{0\le k_1 < \ldots < k_d \le N} \; \sum_{l_1, \ldots, l_d \in \{1, \dots, p-1\}} R^{l_1}_{k_1}(x) \cdot \ldots \cdot R^{l_d}_{k_d}(x) \widehat{\nu}_{l_1p^{k_1} + \ldots + l_dp^{k_d}} = 
\]
\[
\sum_{0\le k_1 < \ldots < k_d \le N}  R^{j_{k_1}}_{k_1}(x) \cdot \ldots \cdot R^{j_{k_d}}_{k_d}(x) =: Q_{j_0, \ldots, j_N}. 
\]
Note that (this decomposition was used in \cite{KazPlo25})
\begin{equation} \label{12}
    Q = \frac{1}{(p-1)^{N+1-d}} \sum_{j_0, \ldots, j_N} Q_{j_0, \ldots, j_N}.
\end{equation}
Applying Young's inequality for convolutions, we obtain:
\begin{equation} \label{ayaya}
    \|Q_{j_0, \ldots, j_N}\|_{{L_{\infty}}} = \|Q * \nu\|_{{L_{\infty}}} \le \|\nu\| \cdot \|Q\|_{L_{\infty}} \le C_1  \|Q\|_{L_{\infty}}.
\end{equation}

Note that, according to \eqref{var-nu}, the constant $C_1$ is independent of the choice of $j_0, \ldots, j_N$. By \cite[Ch. VII, (12.25)]{Blei}, we have for $Q_{j_0, \ldots, j_N}$:
\begin{equation} \label{13}
    \sum_{k_1< \ldots<k_d} \bigl|C_{k_1, \ldots, k_d}^{j_{k_1}, \ldots, j_{k_d}}\bigl|^{\frac{2d}{d+1}} \le C_2 \bigl\|Q_{j_0, \ldots, j_N}\bigl\|_{L_{\infty}}^{\frac{2d}{d+1}},
\end{equation}
where, by Theorem \ref{nadolyche}, the constant $C_2$ is independent of the collection $(j_k)_{k=0}^{\infty}$.

Since formula \eqref{ayaya} is true for any set $j_0, \ldots, j_N$, the following is also true:
\[
\sum_{j_0, \ldots, j_N} \|Q_{j_0, \ldots, j_{N}}\|_{\infty}^{\frac{2d}{d+1}} \le  C_1\sum_{j_0, \ldots, j_N} \|Q\|_{\infty}^{\frac{2d}{d+1}}.
\]
Taking into account \eqref{13},
\[
\sum_{j_0, \ldots, j_N} \sum_{k_1< \ldots k_d} \bigl|C_{k_1, \ldots, k_d}^{j_{k_1}, \ldots, j_{k_d}}\bigl|^{\frac{2d}{d+1}} \le  C_1 C_2 \sum_{j_0, \ldots, j_N} \|Q\|_{\infty}^{\frac{2d}{d+1}}.
\]
Taking into account \eqref{12}, 
\[
\sum_{j_0, \ldots, j_N} \sum_{0<k_1< \ldots <k_d\le N} \bigl|C_{k_1, \ldots, k_d}^{j_{k_1}, \ldots, j_{k_d}}\bigl|^{\frac{2d}{d+1}} = (p-1)^{N+1-d} \sum_{0<k_1< \ldots< k_d\ \le N} \sum_{l_1, \ldots, l_d \in \{1, \dots, p-1\}}  \bigl|C_{k_1, \ldots, k_d}^{l_{1}, \ldots, l_{d}}\bigl|^{\frac{2d}{d+1}}.
\]
On the other hand,
 \[
 \sum_{j_0, \ldots, j_N} \|Q\|_{\infty}^{\frac{2d}{d+1}} = (p-1)^{N+1}\|Q\|_{\infty}^{\frac{2d}{d+1}}
 \]
 Then
 \[
 (p-1)^{N+1-d}\sum_{j_0, \ldots, j_N} \sum_{0<k_1< \ldots< k_d\ \le N} \bigl|C_{k_1, \ldots, k_d}^{j_{k_1}, \ldots, j_{k_d}}\bigl|^{\frac{2d}{d+1}} \le C_1 C_2 (p-1)^{N+1}\|Q\|_{\infty}^{\frac{2d}{d+1}}.
 \]
Hence, we obtain
    \[
            \bigl\|(C^{l_1, \ldots, l_d }_{k_1, \ldots, k_d})_{{0\le k_1 < \ldots < k_d \le N}, l_1, \ldots, l_d \in 1:p-1 } \bigl\|_{l_{2d/(d+1)}} \le 
    \]
        \begin{equation*}
\le C \biggl\|\sum_{0\le k_1 < \ldots < k_d \le N} \; \sum_{l_1, \ldots, l_d \in 1:p-1} C^{l_1, \ldots, l_d }_{k_1, \ldots, k_d} R^{l_1}_{k_1}(x) \cdot \ldots \cdot R^{l_d}_{k_d}(x) \biggl\|_{L_{\infty}}, 
    \end{equation*}
    where $C = (C_1C_2(p-1)^d)^{\frac{d+1}{2d}}$.
\end{proof}


\begin{remark} We note that Theorem \ref{mainS} can be extended from the case of generalized $p$-adic $d$-Rademacher chaos, i.e., the system $A_d$, to the case of the union of $s$-chaos for all $s$ not exceeding $d$.
\end{remark}
\begin{proof}
    Let 
    \[Q = \sum_{0\le k_1 < \ldots < k_s \le N} \; \sum_{l_1, \ldots, l_s \in \{1, \ldots, p-1\}} C^{l_1, \ldots, l_s }_{k_1, \ldots, k_s} R^{l_1}_{k_1}(x) \cdot \ldots \cdot R^{l_s}_{k_s}(x) = 
     \sum\limits_{s=1}^{d} Q^{(s)},\] where $Q^{(s)}$ is a polynomial with respect to the $s$-chaos system. Then, by Lemma \ref{lemma-2},
    \begin{equation} \label{lemma2}
    \|Q^{(s)}\|_{\infty} \le \|Q\|_{\infty}.
    \end{equation}
   By Theorem \ref{mainS}, we can estimate the right-hand side from below \eqref{lemma2}
    \begin{equation} \label{lemma2th2}
   \bigl\|(C_{k_1, \ldots, k_s}^{l_{1}, \ldots, l_{s}})\bigl\|_{\frac{2d}{d+1}} \le \bigl\|(C_{k_1, \ldots, k_s}^{l_{1}, \ldots, l_{s}})\bigl\|_{\frac{2s}{s+1}} \le C\|Q\|_{\infty}.
   \end{equation}
   Summing over $s$ in inequality \eqref{lemma2th2}, we have
    \begin{equation} 
   \sum_{s=1}^{d}\sum_{0<k_1< \ldots< k_s\ \le N} \sum_{l_1, \ldots, l_s \in  \{1, \ldots, p-1\}}  \bigl|C_{k_1, \ldots, k_s}^{l_{1}, \ldots, l_{s}}\bigl|^{\frac{2d}{d+1}} \le Cd\|Q\|_{\infty}^{\frac{2d}{d+1}}.
   \end{equation}
\end{proof}


\begin{thebibliography}{99}
\bibitem{SidonTrig}
S.~Sidon, "Verallgemeinerung eines Satzes über die absolute Konvergenz von Fourierreihen mit Lücken", Math. Ann., 97:1 (1927), 675--676.

\bibitem{Sidon}
S.~Sidon, "Uber orthogonale Entwicklungen", Acta sci. math. Szeged., 10 (1943), 206--253.

\bibitem{KasTem98}
B.~S.~Kashin, V.~N.~Temlyakov, "On a certain norm and related applications", Math. Notes, 64:4 (1998), 551–554.

\bibitem{Rad22}
A.~O.~Radomskii, "Sidon-Type Inequalities and the Space of Quasi-continuous Functions", Proc. Steklov Inst. Math., 319 (2022), 253–264.

\bibitem{Ast-book}
S.~V.~Astashkin, "The Rademacher System in Function Spaces", Fizmatlit, Moscow, 201, [In Russian].
\bibitem{Kashin-Saakyn}
B.~S.~Kashin, A.~A.~Saakyan, "Orthogonal Series", AFC, Moscow, 1999, [In Russian].

\bibitem{Blei}
R.~Blei, "Analysis in integer and fractional dimensions", Cambridge University Press, Cambridge, 2001.

\bibitem{KazPlo25}
A.D. Kazakova, M.G. Plotnikov, “On Lacunarity and Uniqueness for 
$p$-adic Analogs of Rademacher Chaos”, Sib. Math. J., 66, 1184–1194 (2025).
\end{thebibliography}
\end{document}